\documentclass{amsart}
\usepackage{amsfonts}
\usepackage{bbm}
\usepackage{mathrsfs}
\usepackage{amsmath}
\usepackage{bigdelim}
\usepackage{multirow}
\usepackage{amssymb}
\usepackage{indentfirst}
\usepackage{CJK}
\usepackage{fancyhdr}
\usepackage{amscd,amssymb,amsmath,graphicx,verbatim,xy}
\usepackage[TS1,OT1,T1]{fontenc}

\newtheorem{theorem}{Theorem}[section]
\newtheorem{lemma}[theorem]{Lemma}

\newtheorem{proposition}[theorem]{Proposition}

\theoremstyle{definition}
\newtheorem{definition}[theorem]{Definition}

\theoremstyle{remark}
\newtheorem{remark}[theorem]{Remark}
\numberwithin{equation}{section}

\begin{document}

\title[Composition operators on weighted Hardy spaces]{Composition operators on weighted Hardy spaces of polynomial growth}

\author{Bingzhe Hou}
\address{Bingzhe Hou, School of Mathematics, Jilin University, 130012, Changchun, P. R. China}
\email{houbz@jlu.edu.cn}

\author{Chunlan Jiang}
\address{Chunlan Jiang, Department of Mathematics, Hebei Normal University, 050016, Shijiazhuang, P. R. China}
\email{cljiang@hebtu.edu.cn}

\date{}
\thanks{}
\begin{abstract}
In the present paper, we study the composition operators acting on weighted Hardy spaces of polynomial growth, which are concerned with norms, spectra and (semi-)Fredholmness. Firstly, we estimate the norms of the composition operators with symbols of disk automorphisms. Secondly, we discuss the spectra of the composition operators with symbols of disk automorphisms. In particular, it is proven of that the spectrum of a composition operator with symbol of any parabolic disk automorphism is always the unit circle. Thirdly, we consider the Fredholmness of the composition operator $C_{\varphi}$ with symbol $\varphi$ which is an analytic self-map on the closed unit disk. We prove that $C_{\varphi}$ acting on a weighted Hardy space of polynomial growth has closed range (semi-Fredholmness) if and only if $\varphi$ is a finite Blaschke product. Furthermore, it is obtained that $C_{\varphi}$ is Fredholm if and only if $\varphi$ is a disk automorphism.
\end{abstract}
\maketitle
\tableofcontents

\section{Introduction}

Denote by ${\rm Hol}(\mathbb{D})$ and ${\rm Hol}(\overline{\mathbb{D}})$ the space of all analytic functions on the open unit disk $\mathbb{D}$ and the closed unit disk $\overline{\mathbb{D}}$, respectively. Denote by ${\rm Aut}(\mathbb{D})$ the analytic automorphism group on $\mathbb{D}$, which is the set of all analytic bijections from $\mathbb{D}$ to itself. As well known, each $\varphi\in{\rm Aut}(\mathbb{D})$, called disk automorphism or M\"{o}bius transformation, could be written as the following form
$$
\varphi(z)={\rm e}^{\mathbf{i}\theta}\cdot \frac{z_0-z}{1-\overline{z_0}z}, \ \ \ \text{for some} \ z_0\in\mathbb{D}.
$$
For any $f\in {\rm Hol}(\mathbb{D})$, denote the Taylor expansion of $f(z)$ by
$$
f(z)=\sum\limits_{k=0}^{\infty}\widehat{f}(k)z^k.
$$
Let $w=\{w_k\}_{k=1}^{\infty}$ be a sequence of positive numbers. Write $\beta=\{\beta_k\}_{k=0}^{\infty}$,
$$
\beta_0=1, \ \ \text{and} \ \ \beta_k=\prod\limits_{j=1}^{k}w_{j}, \ \ \text{for} \ k\geq1.
$$
The weighted Hardy space $H^2_{\beta}$ induced by the weight sequence $w$ (or $\beta$) is defined by
$$
H^2_{\beta}=\{f(z)=\sum\limits_{k=0}^{\infty}\widehat{f}(k)z^k; \ \sum\limits_{k=0}^{\infty}|\widehat{f}(k)|^{2}{\beta}_k^2<\infty\}.
$$
Moreover, the weighted Hardy space $H^2_{\beta}$ is a complex separable Hilbert space, on which the inner product is defined by,  for any $f,g \in H^2_{\beta}$
$$
\langle  f, g\rangle_{H^2_{\beta}}=\sum\limits_{k=0}^{\infty}{\beta}_k^2\overline{\widehat{g}(k)}\widehat{f}(k)
$$
Then, any $f\in H^2_{\beta}$ equips the following norm
$$
\|f(z)\|_{H^2_{\beta}}=\sqrt{\langle  f, f\rangle_{H^2_{\beta}}}=\sqrt{\sum\limits_{k=0}^{\infty}{\beta}_k^2|\widehat{f}(k)|^{2}}.
$$
One can see that $\{z^n\}_{n=0}^{\infty}$ for each $n\in \mathbb{\mathbb{N}}$. Moreover, we denote by $\|T\|_{H^2_{\beta}}$ the operator norm of the linear operator $T$ acting on $H^2_{\beta}$. In addition, we denote by $\|f\|$ the norm of the function $f$ in $H^2$, and denote by $\|T\|$ the operator norm of the linear operator $T$ acting on $H^2$.

In the present article, we always assume the weight sequence $w$ satisfies the regular condition as follows
$$
\lim\limits_{k\rightarrow\infty}w_k=1.
$$
Then every weighted Hardy space $H^2_{\beta}$ is contained in ${\rm Hol}(\mathbb{D})$.

An analytic self-map $g: \mathbb{D}\rightarrow \mathbb{D}$ introduces a linear operator $C_g: {\rm Hol}(\mathbb{D})\rightarrow {\rm Hol}(\mathbb{D})$, defined by
\[
C_g(f)=f(g),  \ \ \ \text{for any} \ f\in {\rm Hol}(\mathbb{D}).
\]
Then the operator $C_g$ is said to be a composition operator. In this paper, we are interested the composition operators with the symbols of analytic functions on the closed unit disk acting on weighted Hardy spaces of polynomial growth, which are concerned with norms, spectra and (semi-)Fredholmness.

It is a natural problem to study the boundedness or norms of the composition operators $C_g$ acting on weighted Hardy spaces $H^2_{\beta}$. There numerous authors have contributed to it, see \cite{Zor90}, \cite{Kri92}, \cite{C95}, \cite{Mac96}, \cite{H03}, \cite{Bou04}, \cite{H10}, \cite{Gal13} and \cite{Gal16} for instance. We have known that the composition operators induced by some analytic self-maps acting on some weighted Hardy spaces. For instance, the composition operators induced by linear fractional transformations are continuous on the classical Hardy space $H^2$, the weighted Bergman spaces, the weighted Hardy spaces which are larger than $H^2$, the weighted Dirichlet spaces and so on. However, it is still an open problem to determine the boundedness of the composition operators $C_g$ acting on weighted Hardy spaces $H^2_{\beta}$ in general, even for the composition operators induced by disk automorphisms. Recently, inspired by M. Gromov's work in the geometry group theory \cite{G81}, the authors in a previous article \cite{Hou} introduced the notion of weighted Hardy spaces of polynomial growth and proved that each composition operators induced by a disk automorphism acting on a weighted Hardy space of polynomial growth is bounded and the composition operators induced by a disk automorphism acting on a weighted Hardy space of intermediate growth is possible to be unbounded.
\begin{definition}
Let $w=\{w_k\}_{k=1}^{\infty}$  be a sequence of positive numbers with $w_k\rightarrow 1$.
\begin{enumerate}
 \item \ If $\sup\limits_{k}(k+1)|w_k-1|<\infty$, we say that the weighted Hardy space $H^2_{\beta}$ is of polynomial growth.
 \item \ If $\sup\limits_{k}(k+1)|w_k-1|=\infty$, we say that the weighted Hardy space $H^2_{\beta}$ is of intermediate growth.
\end{enumerate}
\end{definition}
\begin{remark}
Notice that many functions spaces such as the classical Hardy space, the weighted Bergman spaces, and the weighted Dirichlet spaces are all the weighted Hardy spaces of polynomial growth.

In addition, the condition $\sup_{k}(k+1)|w_k-1|<\infty$ holds if and only if there exists a positive number $M$ such that for each $k\in\mathbb{N}$,
\[
\frac{k+1}{k+M+1}\leq w_k \leq \frac{k+M+1}{k+1}.
\]
Moreover, such weighted Hardy space $H^2_{\beta}$ is said to be of $M$-polynomial growth.
\end{remark}
In this paper, we focus on the composition operators acting on weighted Hardy spaces of polynomial growth. In Section 2, we will estimate the norms of the composition operators with symbols of disk automorphisms acting on weighted Hardy spaces of polynomial growth.

The research in the spectra of composition operators began formally with Nordgren's paper \cite{Nor}, in which he determined the spectra of invertible composition operators
on the classical Hardy space $H^2$. Then, Kamowitz studied the same problem on $H^p$ and the disk algebra in \cite{Kam75} and \cite{Kam78}, respectively. In the following decades, this topic catched widely attention, we refer to the context by C. Cowen \cite{C95} and the references therein. It is worthy to note that when consider the small space (the weighted Hardy spaces contained in $H^2$ such as Dirichlet spaces and weighted Dirichlet spaces), the effective techniques are different from the case of classical Hardy space or Bergman space. Following from P. Hurst \cite{Hurst97}, E. Gallardo-Guti\'{e}rrez and A. Montes-Rodr\'{\i}guez \cite{Gal04}, W. Higdon \cite{Hig05}, A. Pons \cite{Pons10}, J. Pau and P. P\'{e}rez \cite{Pau}, E. Gallardo-Guti\'{e}rrez and R. Schroderus \cite{Gal16}, one can see the spectra of the composition operators induced by disk automorphisms and linear fractional non-automorphisms acting on the weighted Dirichlet spaces. In Section 3,  we discuss the spectra of the composition operators with symbols of disk automorphisms acting on weighted Hardy spaces of polynomial growth. In particular, it is proven of that the spectrum of a composition operator with symbol of any parabolic disk automorphism is always the unit circle.


Recall that a bounded linear operator $T$ acting on a Hilbert space to itself is called a semi-Fredholm operator, if $T$ has closed range and at least one of $\textrm{dim} \ \textrm{Ker}(T)<\infty$ and $\textrm{dim} \ \textrm{Ker}(T^*)<\infty$ is established. Moreover, $T$ is called a Fredholm operator, if $T$ has closed range and both $\textrm{dim} \ \textrm{Ker}(T)<\infty$ and $\textrm{dim} \ \textrm{Ker}(T^*)<\infty$ are established. Fredholm composition operators on the Hardy spaces of the open unit disk were characterized in \cite{Cima74}, \cite{Bou90} and \cite{Cao97}, Fredholm composition operators on the Bergman space were characterized in \cite{Ak12} and \cite{Bou90}, and Fredholm composition operators on the Dirichlet space were characterized in \cite{Cima76}. And there have been so many other articles to study the Fredholm operators on the spaces of holomorphic functions on domains in $\mathbb{C}$ or $\mathbb{C}^n$, see \cite{Cao05}, \cite{Gali10}, \cite{Kum80}, \cite{Mac97}, \cite{Zor94} and \cite{Zor98} for instance. In particular, G. Cao, L. He and K. Zhu \cite{Cao19} in 2019 proved that for an holomorphic self-map $g$ on a domain $\Omega$ in $\mathbb{C}^n$, if the reproducing kernel $K(w,w)$ of the Hilbert space tends to infinity as $w$ approaches the boundary of $\Omega$, $C_g$ is a Fredholm operator if and only if $C_g$ is a Fredholm operator, if and only if $g$ is a disk automorphism. Recently, G. Cao, L. He and J. Li \cite{Cao22} obtained the same conclusion without the assumption that $K(w,w)\rightarrow\infty$ as $w$ approaches $\partial\Omega$.

Notice that the composition operators induced by non-trivial analytic self-maps acting on weighted Hardy spaces are always injective. Then such composition operator is semi-Fredholm if it has closed range. The closed range composition operators on various function spaces also has attracted wide attention. In 1974, J. Cima, J. Thompson and W. Wogen \cite{Cima74} determined the composition operators on $H^2$ with closed range via the boundary behaviour of the inducing maps. Then, N. Zorboska \cite{Zor94} characterized the closed range composition operators via the properties of the range of the inducing maps on the unit disk, which answered a question in \cite{Cima74}. The closed range composition operators on Bergman space and weighted Bergman spaces were studied by D. Luecking in \cite{L81}, \cite{L84}, and J. Akeroyd and S. Fulmer in \cite{Ak12}, in which one can see that the Navanlinna counting function inducing a reverse Carleson measure is a necessary and sufficient condition for the closed range composition operators on the classical Hardy space and weighted Bergman spaces. However, in the case of the Dirichlet space, this seems to be more difficult. Luecking proved that the Navanlinna counting function inducing a reverse Carleson measure is a necessary condition for the closed range composition operators on the Dirichlet space in \cite{L85}, but is not sufficient condition \cite{L99}. More recently, G. Cao and L. He \cite{Cao23} gave a necessary and sufficient condition for the closed range composition operators on the Dirichlet space.

In Section 4, we consider the Fredholmness of the composition operator $C_{\varphi}$ with symbol of an analytic self-map $\varphi$ on the closed unit disk. We prove that $C_{\varphi}$ acting on a weighted Hardy space of polynomial growth has closed range (semi-Fredholmness) if and only if $\varphi$ is a finite Blaschke product. Furthermore, it is obtained that $C_{\varphi}$ is Fredholm if and only if $\varphi$ is a disk automorphism.

Up to now, there have been numerous research work on the composition operator acting on the classical Hardy space, weighted Bergman spaces and weighted Dirichlet spaces, in which it plays an important role of measure methods (integral representation, Carleson measure, Navanlinna counting function and so on). However, when considering weighted Hardy spaces in general, we may have no appropriate measure to use. The techniques in the present paper involve the operator theory, base theory and some topology as well as the composition operators theory, functional theory and complex analysis.

To avoid confusion, we denote by $(f)^n$ or $(f(z))^n$ the $n$-th power of a function $f$, and denote by $f^n$ the $n$-th iteration of a function $f$ if the iteration makes sense.

\section{Estimations of norms of the composition operators with symbols of disk automorphisms}

In the process of studying the similar representation of analytic functions in \cite{Hou}, the authors obtained the boundedness of the composition operators with symbols of disk automorphisms acting on weighted Hardy spaces of polynomial growth. For an analytic function $g\in{\rm Hol}(\mathbb{D})$, denote by $M_g$ the multiplication operator, which means for any $f\in{\rm Hol}(\mathbb{D})$, $M_g(f(z))=g(z)f(z)$.

\begin{theorem}[Theorem 3.1 in \cite{Hou}]\label{Mz}
Let $H^2_{\beta}$ be the weighted Hardy space of polynomial growth induced by a weight sequence $w=\{w_k\}_{k=1}^{\infty}$.  Then for any $\varphi\in {\textrm Aut(\mathbb{D})}$, $M_z\sim M_{\varphi}$, i.e., $M_z$ is weakly homogeneous on $H^2_{\beta}$. In fact, the composition operator $C_{\varphi}:H^2_{\beta}\rightarrow H^2_{\beta}$ is an isomorphism.
\end{theorem}

Moreover, the composition operators with symbols of finite Blaschke products and analytic self-maps acting on weighted Hardy spaces of polynomial growth were studied.

\begin{theorem}[Theorem 3.3 in \cite{Hou}]\label{CompBHol}
Let $H^2_{\beta}$ be the weighted Hardy space of polynomial growth induced by a weight sequence $w=\{w_k\}_{k=1}^{\infty}$.
If $\psi(z)$ is an analytic function on $\overline{\mathbb{D}}$ with $\psi(\mathbb{D})\subseteq \mathbb{D}$, then $C_{\psi}$ is bounded on $H^2_{\beta}$.
\end{theorem}

\begin{theorem}[Theorem 3.2 in \cite{Hou}]\label{MB}
Let $H^2_{\beta}$ be a weighted Hardy space of polynomial growth, and let $B(z)$ be a finite Blaschke product with order $m$ on $\mathbb{D}$. Then
\[
M_B\sim \bigoplus_{1}^{m}M_z.
\]
\end{theorem}
\begin{remark}\label{BsemiF}
Notice that $M_BC_B=C_BM_z$. Then, following from the above Theorem, one can see that the composition operators with symbols of finite Blaschke products with order $m$ have closed range. Consequently they are semi-Fredholm, since
\[
\textrm{dim} \ \textrm{Ker}(T)=0<\infty \ \ \ \text{and} \ \ \ \textrm{dim} \ \textrm{Ker}(T^*)=\left\{\begin{array}{cc}
0, \ \ \ &\text{if} \ m=1 \\
\infty, \ \ \ &\text{if} \ m>1
\end{array}\right..
\]
\end{remark}

Now, we have known the boundedness of the composition operators induced by disk automorphisms acting on weighted Hardy spaces of polynomial growth. Furthermore, we will give the estimation of the norms of those composition operators. First of all, let us review some notations and basic statements in \cite{Hou}, which are related to the bases theory (we refer to \cite{Sing} and \cite{Nik}).

Let ${H}$ be a complex separable Hilbert space and let $\mathfrak{X}=\{x_n\}_{n=0}^{\infty}$ and $\mathfrak{Y}=\{y_n\}_{n=0}^{\infty}$ be two sequences of vectors in ${H}$.
Define
$$
\langle  \mathfrak{X}, \mathfrak{Y}\rangle_{{H}}=(\langle  x_j, y_i\rangle_{{H}})_{i,j}.
$$
In particular, the Gram matrix of $\mathfrak{X}$ on ${H}$ is defined by
$$
\Gamma_{{H}}(\mathfrak{X})=\langle  \mathfrak{X}, \mathfrak{X}\rangle_{{H}}=(\langle  x_i, x_j\rangle_{{H}})_{i,j}.
$$
Let $\mathfrak{F}=\{f_n\}_{n=0}^{\infty}$ be a sequence of vectors in a weighted Hardy space $H^2_{\beta}$. Then $\mathfrak{F}$ induces a linear operator $X_{\mathfrak{F}}$ on $H^2_{\beta}$, defined by
$$
X_{\mathfrak{F}}(z^n)=f_n(z), \ \ \ \text{for} \ n=0,1,\ldots.
$$
Furthermore, the operator $X_{\mathfrak{F}}$ has a matrix representation under the orthogonal base $\{z^n\}_{n=0}^{\infty}$ as follows
$$
X_{\mathfrak{F}}=\begin{bmatrix}
\widehat{f_0}(0) & \widehat{f_1}(0) & \widehat{f_2}(0)  & \cdots & \widehat{f_k}(0) & \cdots \\
\widehat{f_0}(1)   & \widehat{f_1}(1) & \widehat{f_2}(1)  & \cdots & \widehat{f_k}(1) & \cdots \\
\widehat{f_0}(2)  & \widehat{f_1}(2)   & \widehat{f_2}(2)  & \cdots  & \widehat{f_k}(2) & \cdots \\
\vdots   & \vdots & \vdots & \ddots &\vdots  &\vdots \\
\widehat{f_0}(k)   & \widehat{f_1}(k) & \widehat{f_2}(k) & \cdots & \widehat{f_k}(k) & \cdots \\
\vdots   & \vdots & \vdots&\vdots  &\vdots & \ddots
\end{bmatrix}.
$$
We say that the operator $X_{\mathfrak{F}}$ is the transformation operator of the sequence $\mathfrak{F}$. In particular, if $\mathfrak{F}=\{(g)^n\}_{n=0}^{\infty}$ for some $g\in {\rm Hol}(\mathbb{D})$ with $g(\mathbb{D})\subseteq\mathbb{D}$, then the transformation operator $X_{\mathfrak{F}}$ is just the composition operator $C_g$.

Define the operator $D_{\beta}: H^2_{\beta}\rightarrow H^2$ by
$$
D_{\beta}(z^k)=\beta_kz^k  \ \ \ \text{for} \ n=0,1,\ldots.
$$
As well known, $D_{\beta}$ is an isometry isomorphism. Moreover, $D_{\beta}$ and $D^{-1}_{\beta}$ has matrix representations under the orthogonal base $\{z^n\}_{n=0}^{\infty}$ as follows
$$
D_{\beta}=\begin{bmatrix}
\beta_0 & 0  & \cdots & 0 & \cdots \\
0   & \beta_1   & \cdots & 0 & \cdots \\
\vdots   & \vdots  & \ddots &\vdots  &\vdots \\
0   & 0  & \cdots &\beta_k & \cdots \\
\vdots   & \vdots &\vdots  &\vdots & \ddots
\end{bmatrix}, \ \
D^{-1}_{\beta}=D_{\beta^{-1}}=\begin{bmatrix}
\frac{1}{\beta_0} & 0   & \cdots & 0 & \cdots \\
0   & \frac{1}{\beta_1}  & \cdots & 0 & \cdots \\
\vdots   & \vdots  & \ddots &\vdots  &\vdots \\
0   & 0  & \cdots &\frac{1}{\beta_k} & \cdots \\
\vdots   & \vdots &\vdots  &\vdots & \ddots
\end{bmatrix}.
$$
Then, $D_{\beta}X_{\mathfrak{F}}D^{-1}_{\beta}$ is a linear operator on the classical Hardy space $H^2$.

From now on, we always write $\varphi_{z_0}(z)=\frac{z_0-z}{1-\overline{z_0}z}$, $z_0\in\mathbb{D}\setminus\{0\}$ and $B_0(z)=z\varphi_{z_0}(z)$. Denote
\begin{align*}
&\mathfrak{F}=\{(B_0(z))^n\}_{n=0}^{\infty}, \\
&\mathfrak{F}_{\beta}=\{\frac{(B_0(z))^n}{\beta_n}\}_{n=0}^{\infty}, \\
&\widetilde{\mathfrak{F}}_1=\{(B_0(z))^n, z(B_0(z))^n; n=0,1,\ldots\}, \\
&\widetilde{\mathfrak{F}}_2=\{(B_0(z))^n, \varphi_{z_0}(z)(B_0(z))^n; n=0,1,\ldots\}, \\
&\widetilde{\mathfrak{F}}_{1, \beta}=\{\frac{(B_0(z))^n}{\beta_n}, \frac{z(B_0(z))^n}{\beta_n}; n=0,1,\ldots\}, \\
&\widetilde{\mathfrak{F}}_{2, \beta}=\{\frac{(B_0(z))^n}{\beta_n}, \frac{\varphi_{z_0}(z)(B_0(z))^n}{\beta_n}; n=0,1,\ldots\}.
\end{align*}

It follows from $B_0(z)=z\varphi_{z_0}(z)$ and $\varphi_{z_0}(\varphi_{z_0}(z))=z$ that for every $n=0,1,2,\ldots$,
\begin{align*}
&C_{\varphi_{z_0}}((B_0(z))^n)=(B_0(\varphi_{z_0}(z)))^n=(B_0(z))^n, \\
&C_{\varphi_{z_0}}(z(B_0(z))^n)=\varphi_{z_0}(z)(B_0(\varphi_{z_0}(z)))^n=\varphi_{z_0}(z)(B_0(z))^n.
\end{align*}
Then, the composition operator $C_{\varphi_{z_0}}$ is just a base transformation operator from the Riesz base $\widetilde{\mathfrak{F}}_{1,\beta}$ to the Riesz base $\widetilde{\mathfrak{F}}_{2,\beta}$.
One can see that $C_{B}$ is bounded on $H^2_{\beta}$ and $H^2_{\beta^{-1}}$. Consequently, $D_{\beta}X_{\mathfrak{F}}D^{-1}_{\beta}$ and $D^{-1}_{\beta}X_{\mathfrak{F}}D_{\beta}$ are bounded on $H^2$.

In addition, let
\[
D_{tm}=\begin{bmatrix}
D_1  & 0    & \cdots \\
0   & D_2   & \cdots \\
\vdots &\vdots  &\ddots
\end{bmatrix}=\begin{bmatrix}
1 & w_1^{-1}\overline{z_0} & 0   & 0 & \cdots \\
w_1z_0   & 1 & 0   & 0 & \cdots \\
0   & 0   & 1    & w_2^{-1}\overline{z_0} & \cdots \\
0   & 0 & w_2z_0   & 1 & \cdots \\
\vdots   & \vdots &\vdots  &\vdots & \ddots
\end{bmatrix}.
\]
where
\[
D_n=\begin{bmatrix}
1 & w_n^{-1}\overline{z_0}  \\
w_nz_0   & 1
\end{bmatrix} \ \ \ \ \text{for any} \ n\in\mathbb{N}.
\]

Now, let us begin to estimate the norms of the composition operators induced by disk automorphisms acting on weighted Hardy spaces of polynomial growth.

\begin{theorem}\label{C}
Let $H^2_{\beta}$ be the weighted Hardy space of polynomial growth induced by a weight sequence $w=\{w_k\}_{k=1}^{\infty}$.  Then for any $\varphi\in {\rm Aut(\mathbb{D})}$.
\begin{align}
& (D_{\beta}X^*_{\widetilde{\mathfrak{F}}_1}D^{-1}_{\beta}) (D_{\beta}M^*_{\frac{\sqrt{1-|z_0|^2}}{1-\overline{z_0}z}}D^{-1}_{\beta})(D_{\beta}M_{\frac{\sqrt{1-|z_0|^2}}{1-\overline{z_0}z}}D^{-1}_{\beta})
 (D_{\beta}X_{\widetilde{\mathfrak{F}}_1}D^{-1}_{\beta})=D_{tm}, \label{F1} \\
& (D_{\beta}X^*_{\widetilde{\mathfrak{F}}_2}D^{-1}_{\beta})
 (D_{\beta}X_{\widetilde{\mathfrak{F}}_2}D^{-1}_{\beta})=D_{tm}. \label{F2}
\end{align}
Moreover,
\[
\|C^{-1}_{\varphi}\|_{H^2_{\beta}}=\|C_{\varphi^{-1}}\|_{H^2_{\beta}}=\|C_\varphi\|_{H^2_{\beta}}=\|X_{\widetilde{\mathfrak{F}}_{1}} X^{-1}_{\widetilde{\mathfrak{F}}_{2}}\|_{H^2_{\beta}}\leq
\|X_{\widetilde{\mathfrak{F}}_{1}}\|_{H^2_{\beta}}\| X_{\widetilde{\mathfrak{F}}_{2}}\|_{H^2_{\beta^{-1}}} \|D^{-1}_{tm}\|.
\]
\end{theorem}

\begin{proof}
Each element in $\varphi\in{\rm Aut(\mathbb{D})}$ could be written as follows
$$
\varphi(z)={\rm e}^{\mathbf{i}\theta}\cdot \frac{z_0-z}{1-\overline{z_0}z}={\rm e}^{\mathbf{i}\theta}\cdot\varphi_{z_0}, \ \ \ \text{for some} \ z_0\in\mathbb{D}.
$$
Since $C_{{\rm e}^{\mathbf{i}\theta}z}$ is an isometry, we have
\[
\|C_{\varphi}\|_{H^2_{\beta}}=\|C_{\frac{z_0-z}{1-\overline{z_0}z}}C_{{\rm e}^{\mathbf{i}\theta}z}\|_{H^2_{\beta}}=\|C_{\varphi_{z_0}}\|_{H^2_{\beta}}.
\]
So, it suffices to consider $\varphi_{z_0}(z)$.
Following some computation (the details could be found in the proof of Theorem 3.1 in \cite{Hou}), one can see that the equations (\ref{F1}) and (\ref{F2}) hold.

Since $C_{\varphi_{z_0}}(\widetilde{\mathfrak{F}}_{1})=\widetilde{\mathfrak{F}}_{1}$, we have
\[
C_{\varphi_{z_0}} X_{\widetilde{\mathfrak{F}}_{1}}= X_{\widetilde{\mathfrak{F}}_{2}}.
\]
Together with $\varphi_{z_0}^{-1}=\varphi_{z_0}$, one can see that
\[
\|C^{-1}_{\varphi_{z_0}}\|_{H^2_{\beta}}=\|C_{\varphi_{z_0}^{-1}}\|_{H^2_{\beta}}=\|C_{\varphi_{z_0}}\|_{H^2_{\beta}}=\|X_{\widetilde{\mathfrak{F}}_{1}} X^{-1}_{\widetilde{\mathfrak{F}}_{2}}\|_{H^2_{\beta}}\leq
\|X_{\widetilde{\mathfrak{F}}_{1}}\|_{H^2_{\beta}}\| X^{-1}_{\widetilde{\mathfrak{F}}_{2}}\|_{H^2_{\beta}}.
\]
By the equation (\ref{F2}),
\begin{align*}
\| X^{-1}_{\widetilde{\mathfrak{F}}_{2}}\|_{H^2_{\beta}}=&\|D_{\beta} X^{-1}_{\widetilde{\mathfrak{F}}_{2}}D^{-1}_{\beta}\|   \\
=&\|D^{-1}_{tm}(D_{\beta} X^{*}_{\widetilde{\mathfrak{F}}_{2}}D^{-1}_{\beta})\|   \\
\le&\|D^{-1}_{tm}\|\|D_{\beta} X^{*}_{\widetilde{\mathfrak{F}}_{2}}D^{-1}_{\beta}\|   \\
=&\|D^{-1}_{tm}\|\|D^{-1}_{\beta} X_{\widetilde{\mathfrak{F}}_{2}}D_{\beta}\|   \\
=&\|D^{-1}_{tm}\|\| X_{\widetilde{\mathfrak{F}}_{2}}\|_{H^2_{\beta^{-1}}}.
\end{align*}
This finishes the proof.
\end{proof}

Then, our task is to estimate $\| X_{\widetilde{\mathfrak{F}}_{1}}\|_{H^2_{\beta}}$, $\| X_{\widetilde{\mathfrak{F}}_{2}}\|_{H^2_{\beta^{-1}}}$ and $\|D^{-1}_{tm}\|$.

\begin{lemma}\label{D-1tm}
The notations as above. We have
\[
\|D^{-1}_{tm}\|\leq\frac{\sqrt{1+|z_0|^2\sup\limits_{n}(w_n^2+w_n^{-2})}}{1-|z_0|^2} \ \ \ \text{and} \ \ \|D_{tm}\|\leq \sqrt{1+|z_0|^2\sup\limits_{n}(w_n^2+w_n^{-2})}.
\]
\end{lemma}

\begin{proof}
For each $n\in\mathbb{N}$,
\[
D^*_nD_n=\begin{bmatrix}
1+|w_nz_0|^2 & (w_n+w_n^{-1})\overline{z_0}  \\
(w_n+w_n^{-1})z_0   & 1+|z_0w_n^{-1}|^2
\end{bmatrix}.
\]
It has two eigenvalues as follows
\[
\lambda_1^{(n)}=\frac{1+(w_n^2+w_n^{-2})|z_0|^2-\sqrt{(1+(w_n^2+w_n^{-2})|z_0|^2)^2-4(1-|z_0|^2)^2}}{2},
\]
\[
\lambda_2^{(n)}=\frac{1+(w_n^2+w_n^{-2})|z_0|^2+\sqrt{(1+(w_n^2+w_n^{-2})|z_0|^2)^2-4(1-|z_0|^2)^2}}{2}.
\]
Consequently,
\[
\|D_{n}\|=|\lambda_2^{(n)}|^{\frac{1}{2}}\leq \sqrt{1+|z_0|^2(w_n^2+w_n^{-2})},
\]
\begin{align*}
\|D_{n}^{-1}\|=&|\lambda_1^{(n)}|^{-\frac{1}{2}} \\
=&\left|\frac{2}{1+(w_n^2+w_n^{-2})|z_0|^2-\sqrt{(1+(w_n^2+w_n^{-2})|z_0|^2)^2-4(1-|z_0|^2)^2}}\right|^{\frac{1}{2}} \\
=&\left|\frac{1+(w_n^2+w_n^{-2})|z_0|^2+\sqrt{(1+(w_n^2+w_n^{-2})|z_0|^2)^2-4(1-|z_0|^2)^2}}{2(1-|z_0|^2)^2}\right|^{\frac{1}{2}} \\
\leq& \frac{\sqrt{1+|z_0|^2(w_n^2+w_n^{-2})}}{1-|z_0|^2}.
\end{align*}
Then,
\[
\|D_{tm}\|=\sup\limits_{n}\|D_{n}\|  \leq \sqrt{1+|z_0|^2\sup\limits_{n}(w_n^2+w_n^{-2})},
\]
\[
\|D^{-1}_{tm}\|=\sup\limits_{n}\|D^{-1}_{n}\| \leq\frac{\sqrt{1+|z_0|^2\sup\limits_{n}(w_n^2+w_n^{-2})}}{1-|z_0|^2}.
\]
\end{proof}

\begin{lemma}\label{phiB}
The notations as above. We have
\[
\|X_{{\mathfrak{F}}}\|_{H^2_{\beta}}
\leq\|X_{\widetilde{\mathfrak{F}}_{1}}\|_{H^2_{\beta}} \leq (1+\|M_z\|_{H^2_{\beta}})\|X_{{\mathfrak{F}}}\|_{H^2_{\beta}}
\]
and
\[
\|X_{{\mathfrak{F}}}\|_{H^2_{\beta^{-1}}}
\leq\|X_{\widetilde{\mathfrak{F}}_{2}}\|_{H^2_{\beta^{-1}}}\leq(1+\|M_{\varphi_{z_0}}\|_{H^2_{\beta^{-1}}})\|X_{{\mathfrak{F}}}\|_{H^2_{\beta^{-1}}}.
\]
\end{lemma}

\begin{proof}
One can see that
$$
\Gamma_{H^2_{\beta}}(\widetilde{\mathfrak{F}}_{1, \beta})=\langle  \widetilde{\mathfrak{F}}_{1, \beta}, \widetilde{\mathfrak{F}}_{1, \beta}\rangle_{H^2_{\beta}}=\begin{bmatrix}
\Gamma_{11} &  \Gamma_{12}  \\
\Gamma_{21}   & \Gamma_{22} \\
\end{bmatrix}=\begin{bmatrix}
\langle\mathfrak{F}_{\beta}, \mathfrak{F}_{\beta}\rangle_{H^2_{\beta}} &  \langle\mathfrak{F}_{\beta}, M_z\mathfrak{F}_{\beta}\rangle_{H^2_{\beta}}  \\
\langle M_z\mathfrak{F}_{\beta}, \mathfrak{F}_{\beta}\rangle_{H^2_{\beta}}   & \langle M_z\mathfrak{F}_{\beta}, M_z\mathfrak{F}_{\beta}\rangle_{H^2_{\beta}} \\
\end{bmatrix},
$$
where
\begin{align*}
&\Gamma_{11}=\langle\mathfrak{F}_{\beta}, \mathfrak{F}_{\beta}\rangle_{H^2_{\beta}}=(D^{-1}_{\beta}X^*_{\mathfrak{F}}D_{\beta})(D_{\beta}X_{\mathfrak{F}}D^{-1}_{\beta}), \\
&\Gamma_{12}=\langle\mathfrak{F}_{\beta}, M_z\mathfrak{F}_{\beta}\rangle_{H^2_{\beta}}=(D^{-1}_{\beta}X^*_{\mathfrak{F}}D_{\beta})(D^{-1}_{\beta}M_z^*D_{\beta})(D_{\beta}X_{\mathfrak{F}}D^{-1}_{\beta}),  \\
&\Gamma_{21}=\langle M_z\mathfrak{F}_{\beta}, \mathfrak{F}_{\beta}\rangle_{H^2_{\beta}}=(D^{-1}_{\beta}X^*_{\mathfrak{F}}D_{\beta})(D_{\beta}M_zD^{-1}_{\beta})(D_{\beta}X_{\mathfrak{F}}D^{-1}_{\beta}),   \\
&\Gamma_{22}=\langle M_z\mathfrak{F}_{\beta}, M_z\mathfrak{F}_{\beta}\rangle_{H^2_{\beta}}=
(D^{-1}_{\beta}X^*_{\mathfrak{F}}D_{\beta})(D^{-1}_{\beta}M_z^*D_{\beta})(D_{\beta}M_zD^{-1}_{\beta})(D_{\beta}X_{\mathfrak{F}}D^{-1}_{\beta}).
\end{align*}
Consequently,
\begin{align*}
&\|\Gamma_{11}\|=\|D_{\beta}X_{\mathfrak{F}}D^{-1}_{\beta})\|^2=\|X_{{\mathfrak{F}}}\|_{H^2_{\beta}}^2, \\
&\|\Gamma_{12}\|\le\|D^{-1}_{\beta}X^*_{\mathfrak{F}}D_{\beta}\|\|D^{-1}_{\beta}M_z^*D_{\beta}\|\|(D_{\beta}X_{\mathfrak{F}}D^{-1}_{\beta}\|
=\|M_z\|_{H^2_{\beta}}\|X_{{\mathfrak{F}}}\|_{H^2_{\beta}}^2,  \\
&\|\Gamma_{21}\|\le\|D^{-1}_{\beta}X^*_{\mathfrak{F}}D_{\beta}\|\|D_{\beta}M_zD^{-1}_{\beta}\|\|D_{\beta}X_{\mathfrak{F}}D^{-1}_{\beta}\|
=\|M_z\|_{H^2_{\beta}}\|X_{{\mathfrak{F}}}\|_{H^2_{\beta}}^2,   \\
&\|\Gamma_{22}\|\le\|D^{-1}_{\beta}X^*_{\mathfrak{F}}D_{\beta}\|\|D_{\beta}M_zD^{-1}_{\beta}\|^2\|D_{\beta}X_{\mathfrak{F}}D^{-1}_{\beta}\|
=\|M_z\|_{H^2_{\beta}}^2\|X_{{\mathfrak{F}}}\|_{H^2_{\beta}}^2.
\end{align*}
Then,
\[
\|X_{{\mathfrak{F}}}\|_{H^2_{\beta}}^2=\|\Gamma_{11}\|
\le\|\Gamma_{H^2_{\beta}}(\widetilde{\mathfrak{F}}_{1, \beta})\|=\|X_{\widetilde{\mathfrak{F}}_{1}}\|_{H^2_{\beta}}^2
\]
and
\begin{align*}
\|X_{\widetilde{\mathfrak{F}}_{1}}\|_{H^2_{\beta}}^2=\Gamma_{H^2_{\beta}}(\widetilde{\mathfrak{F}}_{1, \beta})
\le&\|\Gamma_{11}\|+\|\Gamma_{12}\|+\|\Gamma_{21}\|+\|\Gamma_{22}\| \\
\leq& (1+\|M_z\|_{H^2_{\beta}})^2\|X_{{\mathfrak{F}}}\|_{H^2_{\beta}}^2.
\end{align*}
Therefore,
\[
\|X_{{\mathfrak{F}}}\|_{H^2_{\beta}}
\leq\|X_{\widetilde{\mathfrak{F}}_{1}}\|_{H^2_{\beta}} \leq (1+\|M_z\|_{H^2_{\beta}})\|X_{{\mathfrak{F}}}\|_{H^2_{\beta}},
\]
and similarly, we have
\[
\|X_{{\mathfrak{F}}}\|_{H^2_{\beta^{-1}}}
\leq\|X_{\widetilde{\mathfrak{F}}_{2}}\|_{H^2_{\beta^{-1}}}\leq(1+\|M_{\varphi_{z_0}}\|_{H^2_{\beta^{-1}}})\|X_{{\mathfrak{F}}}\|_{H^2_{\beta^{-1}}}.
\]
\end{proof}

Recall that for $\lambda>0$, the weighted Dirichlet space $D_{\lambda}^2$ consists of those analytic functions $f$ on $\mathbb{D}$ with
\[
\|f\|_{D_{\lambda}^2}=\left(|f(0)|^2+ \int_{\mathbb{D}}|f'(z)|^2(1-|z|^2)^{\lambda-1}\textrm{d}A(z) \right)^{\frac{1}{2}}<\infty
\]
where $\textrm{d}A(z)=\frac{1}{\pi}\textrm{d}x\textrm{d}y$ is the normalized area measure on $\mathbb{D}$. As well known,
the weighted Dirichlet space $D_{\lambda}^2$ could be seemed as a weighted Hardy space $H^2_{\beta^{(\lambda)}}$, where $\beta^{(\lambda)}_0=1$ and for $k=1,2,\ldots$, $\beta_k^{(\lambda)}=\prod_{j=1}^{k}w_{j}^{(\lambda)}$ with $w_{j}^{(\lambda)}=\sqrt{\frac{j+2\lambda+1}{j+1}}$.

For $N\in\mathbb{N}$, let $\beta^{(N,1)}=\{\beta^{(N,1)}_k\}_{k=0}^{\infty}$ and $\beta^{(N,2)}=\{\beta^{(N,2)}_k\}_{k=0}^{\infty}$, where
\[
w_{j}^{(N,1)}=\frac{j+N}{j}, \ \ \beta_k^{(N,1)}=\prod_{j=1}^{k}w_{j}^{(N,1)}=C_{N+k}^{k}=\frac{(N+k)!}{N!\cdot k!},
\]
and
\[
w_{j}^{(N,2)}=\frac{(j+1)^N}{j^N}, \ \ \beta_k^{(N,2)}=\prod_{j=1}^{k}w_{j}^{(N,2)}=(k+1)^N.
\]
Furthermore, denote by $D_{N,1}^2$ and $D_{N,2}^2$ the weighted Hardy space $H^2_{\beta^{(N,1)}}$ and $H^2_{\beta^{(N,2)}}$, respectively. It is not difficult to see that both $D_{N,1}^2$ and $D_{N,2}^2$ are the weighted Dirichlet space $D_{N}^2$ with equivalent norms.

\begin{lemma}\label{Mnorm}
Let $H^2_{\beta}$ be a weighted Hardy space of $N$-polynomial growth induced by the weight sequence $\{\beta_n\}_{n=1}^{\infty}$. Then
\[
\|M_{z^n}\|_{H^2_{\beta}}\leq C_{N+k}^{k}=\beta^{(N,1)}_n, \ \ \ \|M_{\frac{1}{1-\overline{z_0}z}}\|_{H^2_{\beta}}\leq \frac{1}{(1-|z_0|)^{N+1}},
\]
and
\[
\|M_{\frac{z_0-z}{1-\overline{z_0}z}}\|_{H^2_{\beta}}\leq \frac{N+2}{(1-|z_0|)^{N+1}}.
\]
\end{lemma}

\begin{proof}
Since $w_n\le\frac{n+1+N}{n+1}<\frac{n+N}{n}=w_n^{(N,1)}$ for all $n=1,2,\ldots$, and $\{w_n^{(N,1)}\}$ is a decreasing sequence of positive numbers, it is easy to see that
\[
\|M_{z^n}\|_{H^2_{\beta}}\le\|M_{z^n}\|_{D_{N}^2}=C_{N+n}^{n}=\beta^{(N)}_n.
\]
Furthermore,
\[
\|M_{\frac{1}{1-\overline{z_0}z}}\|_{H^2_{\beta}}=
\|\sum\limits_{n=0}^{\infty}z_0^nM_{z^n}\|_{H^2_{\beta}}\le\sum\limits_{n=0}^{\infty}|z_0|^n \|M_{z^n}\|_{H^2_{\beta}}\le\sum\limits_{n=0}^{\infty}|z_0|^n\beta^{(N,1)}_n.
\]
Then
\[
\frac{1}{(1-|z_0|)^{N+1}}=\sum\limits_{n=0}^{\infty}|z_0|^nC_{N+n}^{n}=\sum\limits_{n=0}^{\infty}|z_0|^n\beta^{(N,1)}_n\ge\|M_{\frac{1}{1-\overline{z_0}z}}\|_{H^2_{\beta}}
\]
and consequently,
\[
\|M_{\frac{z_0-z}{1-\overline{z_0}z}}\|_{H^2_{\beta}}\leq(|z_0|+\|M_{z}\|_{H^2_{\beta}})\|M_{\frac{1}{1-\overline{z_0}z}}\|_{H^2_{\beta}}\le \frac{N+2}{(1-|z_0|)^{N+1}}.
\]
This finishes the proof.
\end{proof}

Define $D_w, D:H^2_{\beta}\rightarrow H^2_{\beta}$ by, for any $f(z)\in H^2_{\beta}$
$$
D_wf(z)=\sum\limits_{k=0}^{\infty}w_{k+1}\widehat{f}(k)z^k  \ \ \text{and} \ \
Df(z)=\sum\limits_{k=0}^{\infty}\frac{k+2}{k+1}\cdot \widehat{f}(k)z^k .
$$
We could write the two operators $D_w$ and $D$ in matrix form under the orthogonal base $\{z^k\}_{k=0}^{\infty}$ as follows,
$$
D_w=\begin{bmatrix}
w_1 & 0 & 0  & \cdots & 0 & \cdots \\
0   & w_2 & 0  & \cdots & 0 & \cdots \\
0   & 0   & w_3  & \cdots  & 0 & \cdots \\
\vdots   & \vdots & \vdots & \ddots &\vdots  &\vdots \\
0   & 0 & 0 & \cdots &w_k & \cdots \\
\vdots   & \vdots & \vdots&\vdots  &\vdots & \ddots
\end{bmatrix} \ \ \text{and}  \ \
D=\begin{bmatrix}
2 & 0 & 0  & \cdots & 0 & \cdots \\
0   & \frac{3}{2} & 0  & \cdots & 0 & \cdots \\
0   & 0   & \frac{4}{3}  & \cdots  & 0 & \cdots \\
\vdots   & \vdots & \vdots & \ddots &\vdots  &\vdots \\
0   & 0 & 0 & \cdots \frac{k+1}{k} & \cdots \\
\vdots   & \vdots & \vdots&\vdots  &\vdots & \ddots
\end{bmatrix},
$$
Let $\widetilde{\beta}_n=(n+1)\beta_n$ for $n=0,1,\ldots$ and $\widetilde{w}_k=\frac{\beta_{k}}{\beta_{k-1}}=w_k\cdot\frac{k+1}{k}$ for $k=1,2,\ldots$. Moreover, define $D_{\widetilde{w}}:H^2_{\beta}\rightarrow H^2_{\beta}$ by, for any $f(z)\in H^2_{\beta}$
$$
D_wf(z)=\sum\limits_{k=0}^{\infty}w_{k+1}\widehat{f}(k)z^k,
$$
which could be written in a matrix form under the orthogonal base $\{z^k\}_{k=0}^{\infty}$ as follows,
$$
D_{\widetilde{w}}=DD_w=\begin{bmatrix}
\widetilde{w_1} & 0 & 0  & \cdots & 0 & \cdots \\
0   & \widetilde{w_2} & 0  & \cdots & 0 & \cdots \\
0   & 0   & \widetilde{w_3}  & \cdots  & 0 & \cdots \\
\vdots   & \vdots & \vdots & \ddots &\vdots  &\vdots \\
0   & 0 & 0 & \cdots &w_k & \cdots \\
\vdots   & \vdots & \vdots&\vdots  &\vdots & \ddots
\end{bmatrix}.
$$

\begin{lemma}\label{n+1}
Let $H^2_{\beta}$ be the weighted Hardy space induced by a weight sequence $w=\{w_k\}_{k=1}^{\infty}$ with $w_k\rightarrow 1$.
Denote $\widetilde{\beta}=\{\widetilde{\beta}_n\}_{n=0}^{\infty}$, where $\widetilde{\beta}_n=(n+1)\beta_n$. Then,
\[
\|C_{B_0}\|_{H^2_{\widetilde{\beta}}}\leq \|D_{\widetilde{w}}\| \cdot \|M_{B'_0}\|_{H^2_{\beta}}\cdot \|C_{B_0}\|_{H^2_{\beta}}.
\]
\end{lemma}

\begin{proof}
Following from the proof of Proposition 2.15 in \cite{Hou}, one can see that
\[
\langle  \frac{(B_{0})^{i+1}}{\widetilde{\beta_i}}, \frac{(B_{0})^{j+1}}{\widetilde{\beta_j}}\rangle_{H^2_{\widetilde{\beta}}}=
\langle  DD_wM_{B'_{0}}\left(\frac{(B_{0})^{i}}{\beta_i}\right), DD_wM_{B'_{0}}\left(\frac{(B_{0})^{j}}{\beta_j}\right) \rangle_{H^2_{\beta}},
\]
and then
\[
\Gamma_{H^2_{\widetilde{\beta}}}\left(\left\{\frac{(B_{0})^{n+1}}{\widetilde{\beta}_n}\right\}_{n=0}^{\infty}\right)
=\Gamma_{H^2_{\beta}}\left(\left\{DD_wM_{B'_{0}}\left(\frac{(B_{0})^{n}}{\beta_n}\right)\right\}_{n=0}^{\infty}\right).
\]
Notice that
\[
\Gamma_{H^2_{\widetilde{\beta}}}\left(\left\{\frac{(B_{0})^{n}}{\widetilde{\beta}_n}\right\}_{n=0}^{\infty}\right)=\begin{bmatrix}
1 & 0 \\
0   & \Gamma_{H^2_{\widetilde{\beta}}}(\{\frac{B_{0}^{n}}{\widetilde{\beta}_n}\}_{n=1}^{\infty})
\end{bmatrix}=\begin{bmatrix}
1 & 0 \\
0   & D^{-1}_{\widetilde{w}}\Gamma_{H^2_{\widetilde{\beta}}}(\{\frac{B_{0}^{n+1}}{\widetilde{\beta}_n}\}_{n=0}^{\infty})D^{-1}_{\widetilde{w}}.
\end{bmatrix}
\]
Then,
\begin{align*}
\|X_{\mathfrak{F}}\|^2_{H^2_{\widetilde{\beta}}}
=&\left\|\Gamma_{H^2_{\widetilde{\beta}}}\left(\left\{\frac{(B_{0})^{n}}{\widetilde{\beta}_n}\right\}_{n=0}^{\infty}\right)\right\| \\
=&\max\left\{1,\left\|D^{-1}_{\widetilde{w}}\Gamma_{H^2_{\widetilde{\beta}}}\left(\left\{
\frac{(B_{0})^{n+1}}{\widetilde{\beta}_n}\right\}_{n=0}^{\infty}\right)D^{-1}_{\widetilde{w}}\right\|\right\}.
\end{align*}
Since
\begin{align*}
\left\|D^{-1}_{\widetilde{w}}\Gamma_{H^2_{\widetilde{\beta}}}\left(\left\{\frac{(B_{0})^{n+1}}{\widetilde{\beta}_n}\right\}_{n=0}^{\infty}\right)D^{-1}_{\widetilde{w}}\right\|
\le&\left\|\Gamma_{H^2_{\widetilde{\beta}}}\left(\left\{\frac{(B_{0})^{n+1}}{\widetilde{\beta}_n}\right\}_{n=0}^{\infty}\right)\right\| \\
=&\left\|\Gamma_{H^2_{\beta}}\left(\left\{DD_wM_{B'_{0}}\left(\frac{(B_{0})^{n}}{\beta_n}\right)\right\}_{n=0}^{\infty}\right)\right\| \\
=&\|D_{\widetilde{w}}M_{B'_{0}}X_{\mathfrak{F}}\|^2_{H^2_{\beta}} \\
\le&(\|D_{\widetilde{w}}\|_{H^2_{\beta}} \|M_{B'_{0}}\|_{H^2_{\beta}} \|X_{\mathfrak{F}}\|_{H^2_{\beta}})^2 \\
=&(\|D_{\widetilde{w}}\| \|M_{B'_{0}}\|_{H^2_{\beta}} \|C_{B_0}\|_{H^2_{\beta}})^2,
\end{align*}
we have
\begin{align*}
\|C_{B_0}\|_{H^2_{\widetilde{\beta}}}=&\|X_{\mathfrak{F}}\|_{H^2_{\widetilde{\beta}}} \\
\le&\max\{1,\|D_{\widetilde{w}}\| \|M_{B'_{0}}\|_{H^2_{\beta}} \|C_{B_0}\|_{H^2_{\beta}}\} \\
=&\|D_{\widetilde{w}}\| \|M_{B'_{0}}\|_{H^2_{\beta}} \|C_{B_0}\|_{H^2_{\beta}}.
\end{align*}
\end{proof}

The following result of C. Cowen \cite{C90} is also useful in our estimation.
\begin{theorem}[Theorem 7 in \cite{C90}]\label{Controll}
Suppose that $H^2_{\beta}$ and $H^2_{\beta'}$ are weighted Hardy spaces and
$$
w_{k+1}=\frac{\beta_{k+1}}{\beta_k}\geq \frac{\beta'_{k+1}}{\beta'_k}=w'_{k+1}, \ \ \ \text{for} \ k=0,1,2,\ldots.
$$
Let $\psi(z)$ be an analytic function on $\mathbb{D}$ with $\psi(\mathbb{D})\subseteq \mathbb{D}$ and $\psi(0)=0$. If $C_{\psi}$ is bounded on $H^2_{\beta}$, then $C_{\psi}$ is bounded on $H^2_{\beta'}$ and $\|C_{\psi}\|_{H^2_{\beta}}\geq\|C_{\psi}\|_{H^2_{\beta'}}$.
\end{theorem}

Now, based on all of the previous preliminaries, we can estimate the norms of the composition operators induced by $B_0$ and disk automorphisms acting on a weighted Hardy space of $N$-polynomial growth.

\begin{proposition}\label{NormB}
Let $H^2_{\beta}$ be a weighted Hardy space of $N$-polynomial growth. Then
\[
\|C_{B_0}\|_{H^2_{\beta}}\le\frac{K_2K_1^N\prod_{j=1}^{N}(j^3+5j^2+2j)}{(1-|z_0|)^{N(N+1)}},
\]
where $K_1$ and $K_2$ are positive constants.
\end{proposition}
\begin{proof}
Obviously,
\[
B'_{0}(z)=\left( z\cdot\frac{z_0-z}{1-\overline{z_0}z}\right)'=\frac{z_0-2z+\overline{z_0}z^2}{(1-\overline{z_0}z)^2}.
\]
By Lemma \ref{Mnorm}, for $j=1,2,\ldots,N$,
\begin{align*}
\|M_{B'_{0}}\|_{D^2_{j-1,1}}\leq&(|z_0|+2\|M_{z}\|_{D^2_{j-1,1}}+|z_0|\|M_{z^2}\|_{D^2_{j-1,1}})\|M_{\frac{1}{1-\overline{z_0}z}}\|^2_{D^2_{j-1,1}} \\
\le& \frac{1+2j+j(j+1)/2}{(1-|z_0|)^{2j}} \\
=&\frac{j^2+5j+2}{2(1-|z_0|)^{2j}}.
\end{align*}
Moreover, since both $D_{N,1}^2$ and $D_{N,2}^2$ are the weighted Dirichlet space $D_{N}^2$ with equivalent norm, there exists constants $K_1>0$ and $K_2>0$ such that
\[
\|M_{B'_{0}}\|_{D^2_{j-1,2}}\le K_1\|M_{B'_{0}}\|_{D^2_{j-1,1}}\le \frac{K_1(j^2+5j+2)}{2(1-|z_0|)^{2j}}
\]
and
\[
\|C_{B_0}\|_{D^2_{N,1}}\le K_2\|C_{B_0}\|_{D^2_{N,2}}.
\]
In addition, for $j=1,2,\ldots,N$,
\[
\|DD_{w^{(j-1,1)}}\|=2j.
\]
Then, by Lemma \ref{n+1} and $C_{B_0}$ being an isometry on the classical Hardy space,
\begin{align*}
\|C_{B_0}\|_{D^2_{N,2}}\le&  \prod\limits_{j=1}^{N}\|DD_{w^{(j-1,1)}}\| \cdot  \prod\limits_{j=1}^{N}\|M_{C_{B'_0}}\|_{D^2_{j-1,2}}\cdot \|C_{B_0}\|_{D^2_{0,2}} \\
\le&\prod\limits_{j=1}^{N}(2j)\cdot\prod\limits_{j=1}^{N}\frac{K_1(j^2+5j+2)}{2(1-|z_0|)^{2j}}\cdot\|C_{B_0}\| \\
=&\frac{K_1^N\prod_{j=1}^{N}(j^3+5j^2+2j)}{(1-|z_0|)^{N(N+1)}}.
\end{align*}
Thus, it follows from Theorem \ref{Controll} that
\[
\|C_{B_0}\|_{H^2_{\beta}}\le\|C_{B_0}\|_{D^2_{N,1}}\le K_2\|C_{B_0}\|_{D^2_{N,2}}\le\frac{K_2K_1^N\prod_{j=1}^{N}(j^3+5j^2+2j)}{(1-|z_0|)^{N(N+1)}}.
\]
\end{proof}

\begin{theorem}\label{NormC}
Let $H^2_{\beta}$ be a weighted Hardy space of $N$-polynomial growth induced by the weight sequence $w=\{w_k\}_{k=1}^{\infty}$.  Then for any $\varphi\in {\rm Aut(\mathbb{D})}$, there exists a positive number $K$ such that
\[
\|C_\varphi\|_{H^2_{\beta}}\leq \frac{K}{(1-|z_0|)^{2N^2+3N+2}}.
\]
\end{theorem}

\begin{proof}
By Theorem \ref{C},
\[
\|C_\varphi\|_{H^2_{\beta}}\leq
(1+\|M_{z}\|_{H^2_{\beta}})(1+\|M_{\varphi}\|_{H^2_{\beta^{-1}}})\|X_{\mathfrak{F}}\|_{H^2_{\beta}} \|X_{\mathfrak{F}}\|_{H^2_{\beta^{-1}}} \|D^{-1}_{tm}\|,
\]
it suffices to estimate the items in the right of the above inequality.
Following from Lemma \ref{Mnorm}, one can see
\[
\|M_{z}\|_{H^2_{\beta}}\le N+1 \ \ \ \text{and} \ \ \ \|M_{\frac{z_0-z}{1-\overline{z_0}z}}\|_{H^2_{\beta}}\le \frac{N+2}{(1-|z_0|)^{N+1}}.
\]
By Proposition \ref{NormB}, we have
\[
\|X_{\mathfrak{F}}\|_{H^2_{\beta}}=\|C_{B_0}\|_{H^2_{\beta}}\le\frac{K_2K_1^N\prod_{j=1}^{N}(j^3+5j^2+2j)}{(1-|z_0|)^{N(N+1)}}
\]
and
\[
\|X_{\mathfrak{F}}\|_{H^2_{\beta^{-1}}}=\|C_{B_0}\|_{H^2_{\beta^{-1}}}\le \frac{K_2K_1^N\prod_{j=1}^{N}(j^3+5j^2+2j)}{(1-|z_0|)^{N(N+1)}}.
\]
Notice that
\[
\sqrt{1+|z_0|^2\sup\limits_{n}(w_n^2+w_n^{-2})}\le\sqrt{\sup\limits_{n}(1+w_n^2+w_n^{-2})}\le\sqrt{\sup\limits_{n}(w_n+w_n^{-1})^{2}}\le N+2.
\]
Then by Lemma \ref{D-1tm},
\[
\|D^{-1}_{tm}\|=\sup\limits_{n}\|D^{-1}_{n}\| \leq\frac{\sqrt{1+|z_0|^2\sup\limits_{n}(w_n^2+w_n^{-2})}}{1-|z_0|^2}\le\frac{N+2}{1-|z_0|^2}\le\frac{N+2}{1-|z_0|}.
\]
Therefore, if let
\[
K=K^2_2K_1^{2N}(N+2)^2(N+3)\prod_{j=1}^{N}(j^3+5j^2+2j)^2,
\]
we obtain that
\begin{align*}
\|C_\varphi\|_{H^2_{\beta}}\leq&
(1+\|M_{z}\|_{H^2_{\beta}})(1+\|M_{\varphi}\|_{H^2_{\beta^{-1}}})\|X_{\mathfrak{F}}\|_{H^2_{\beta}} \|X_{\mathfrak{F}}\|_{H^2_{\beta^{-1}}} \|D^{-1}_{tm}\| \\
\le& (N+2)\left(1+\frac{N+2}{(1-|z_0|)^{N+1}}\right)\left(\frac{K_2K_1^N\prod_{j=1}^{N}(j^3+5j^2+2j)}{(1-|z_0|)^{N(N+1)}} \right)^2\frac{N+2}{1-|z_0|} \\
\le& \frac{K^2_2K_1^{2N}(N+2)^2(N+3)\prod_{j=1}^{N}(j^3+5j^2+2j)^2}{(1-|z_0|)^{2N^2+3N+2}} \\
\le& \frac{K}{(1-|z_0|)^{2N^2+3N+2}}.
\end{align*}
\end{proof}

Obviously, if $H^2_{\beta}$ is a weighted Hardy space of $M$-polynomial growth for some $M>0$, then for any positive integer $N$ more than $M$, $H^2_{\beta}$ is also a weighted Hardy space of $N$-polynomial growth. Thus, we could estimate the norms of the composition operators induced by disk automorphisms acting on any weighted Hardy space of polynomial growth.

\section{Spectra of composition operators with symbols of disk automorphisms}

In this section, we study the spectra of composition operators with symbols of disk automorphisms acting on weighted Hardy spaces of polynomial growth.
Every element of $\textrm{Aut}(\mathbb{D})$ has two fixed points (counting multiplicity) in $\mathbb{C}$, and thus can be classified by the location of the fixed points.

\begin{definition}
For an element $\varphi$ in $\textrm{Aut}\mathbb{D}$,
\begin{enumerate}
\item \ $\varphi$ is said to be elliptic, if it has one fixed point in $\mathbb{D}$ and one in the complement of $\overline{\mathbb{D}}$;
\item \ $\varphi$ is said to be parabolic, if it has just one fixed point on the unit circle $\partial\mathbb{D}$ (of multiplicity $2$);
\item \ $\varphi$ is said to be hyperbolic, if it has two distinct fixed points on $\partial\mathbb{D}$.
\end{enumerate}
\end{definition}

Two disk automorphisms $\varphi$ and  $\psi$ are called conformally equivalent, if there exists a disk automorphism $\tau$ such that
\[
\psi=\tau\circ\varphi\circ\tau^{-1}.
\]
The main advantage of conformal equivalence is in the placement of the fixed points.
Every elliptic disk automorphism is conformally equivalent to one whose fixed point in $\mathbb{D}$ is $0$.

\begin{lemma}\label{elliptic}
Let $\varphi$ be an elliptic disk automorphism with fixed point $\alpha$ in $\mathbb{D}$. Then
$\varphi$ is conformally equivalent to $\psi(z)=\lambda z$, where $\lambda=\varphi'(\alpha)$.
\end{lemma}

This result is well-known, and one can find a proof in \cite{A2016} (Lemma 2.1 in it).

Every parabolic disk automorphism is conformally equivalent to one whose fixed point (of multiplicity $2$) is $1$.
\begin{lemma}[Lemma 4.1.2 in \cite{Pons2007}]\label{parabolic}
Let $\varphi$ be a parabolic disk automorphism. Then $\varphi$ is conformally equivalent to either
\[
\psi_1(z)=\frac{(1+\mathbf{i})z-1}{z+\mathbf{i}-1} \ \ \ \text{or} \ \ \ \psi_2(z)=\frac{(1-\mathbf{i})z-1}{z-\mathbf{i}-1}.
\]
\end{lemma}
\begin{remark}\label{reparabolic}
After a straight computation, one can see
\[
\psi_1(z)= \psi_2^{-1}(z),
\]
and for any positive integer $n$,
\[
\psi_1^n(z)=\underset{n}{\underbrace{\psi_1\circ\cdots\circ\psi_1}}(z)=\frac{(n+\mathbf{i})z-n}{nz+\mathbf{i}-n},
\]
and
\[
\psi_2^n(z)=\underset{n}{\underbrace{\psi_2\circ\cdots\circ\psi_2}}(z)=\frac{(n-\mathbf{i})z-n}{nz-\mathbf{i}-n}.
\]
Moreover, for any positive integer $n$, $\psi_1^n(z)$ is conformally equivalent to $\psi_1(z)$, and $\psi_2^n(z)$ is conformally equivalent to $\psi_2(z)$. In fact, let
\[
\tau_n(z)=\frac{(n+1)z+(n-1)}{(n-1)z+(n+1)},
\]
then
\[
\tau_n\circ\psi_1\circ\tau_n^{-1}=\psi_1^n \ \ \ \text{and} \ \ \ \tau_n\circ\psi_2\circ\tau_n^{-1}=\psi_2^n.
\]
\end{remark}

Every hyperbolic disk automorphism is conformally equivalent to one whose fixed points in $\partial\mathbb{D}$ are $1$ and $-1$.

\begin{lemma}[Theorem 6 in \cite{Nor}]\label{hyperboic}
Let $\varphi$ be a hyperbolic disk automorphism. Then, $\varphi$ is conformally equivalent to
\[
\psi(z)=\frac{z+r}{1+rz},
\]
where $r=|\varphi(0)|$.
\end{lemma}

\begin{lemma}\label{similar}
Let $H^2_{\beta}$ be a weighted Hardy space of polynomial growth. If $\varphi$ and  $\psi$ are conformally equivalent disk automorphisms, then the induced composition operators $C_{\varphi}$ and $C_{\psi}$ acting on $H^2_{\beta}$ are similar.
\end{lemma}
\begin{proof}
Since $\varphi$ and  $\psi$ are conformally equivalent, there exists a disk automorphism $\tau$ such that
\[
\psi=\tau\circ\varphi\circ\tau^{-1}.
\]
By Theorem \ref{C}, $C_{\tau}$ is an invertible bounded operator on $H^2_{\beta}$. For any $f\in H^2_{\beta}$, we have
\[
C_\tau^{-1}\circ C_\varphi\circ C_\tau (f) =C_{\tau^{-1}}\circ C_\varphi\circ C_\tau (f) =C_{\tau\circ\varphi\circ\tau^{-1}}(f)=C_\psi(f),
\]
then $C_{\varphi}$ and $C_{\psi}$ acting on $H^2_{\beta}$ are similar.
\end{proof}

\begin{theorem}\label{Se}
Let $H^2_{\beta}$ be a weighted Hardy space of polynomial growth.  If $\varphi$ is an elliptic disk automorphism with fixed point $\alpha$ in $\mathbb{D}$, then $C_{\varphi}$ acting on $H^2_{\beta}$ is a diagonal operator and hence it is a normal operator. Moreover, the spectrum of $C_{\varphi}$ acting on $H^2_{\beta}$ is as follows
\[
\sigma(C_{\varphi})=\overline{\{(\varphi'(\alpha))^n; \ n=0,1,2,\ldots\}}.
\]
\end{theorem}

\begin{proof}
By Lemma \ref{elliptic},
$\varphi$ is conformally equivalent to $\psi(z)=\lambda z$, where $\lambda=\varphi'(\alpha)$. Since the composition operator induced by a disk automorphism is an invertible bounded operator on $H^2_{\beta}$, we have $C_{\varphi}$ is similar to $C_{\psi}$ acting on $H^2_{\beta}$. Notice that the composition operator $C_{\psi}$ is a diagonal operator, more precisely, it has the following matrix presentation under the orthogonal base $\{\frac{z^n}{\beta_n}\}_{k=1}^{\infty}$
$$
C_{\psi}=\begin{bmatrix}
1 & 0 & 0  & \cdots & 0 & \cdots \\
0   & \lambda & 0  & \cdots & 0 & \cdots \\
0   & 0   & \lambda^2  & \cdots  & 0 & \cdots \\
\vdots   & \vdots & \vdots & \ddots &\vdots  &\vdots \\
0   & 0 & 0 & \cdots \lambda^{n} & \cdots \\
\vdots   & \vdots & \vdots&\vdots  &\vdots & \ddots
\end{bmatrix},
$$
Then $C_{\varphi}$ acting on $H^2_{\beta}$ is a diagonal operator and hence it is a normal operator. Moreover,
\[
\sigma(C_{\varphi})=\sigma(C_{\psi})=\overline{\{(\varphi'(\alpha))^n; \ n=0,1,2,\}}.
\]
\end{proof}

\begin{theorem}\label{Sp}
Let $H^2_{\beta}$ be a weighted Hardy space of polynomial growth.  If $\varphi$ is a parabolic disk automorphism, then the spectrum of $C_{\varphi}$ acting on $H^2_{\beta}$ is the unit circle.
\end{theorem}

\begin{proof}
Suppose that $H^2_{\beta}$ is a weighted Hardy space of $N$-polynomial growth.
By Lemma \ref{parabolic}, Remark \ref{reparabolic} and Lemma \ref{similar}, it suffices to consider the spectra of
\[
\psi(z)=\frac{(1+\mathbf{i})z-1}{z+\mathbf{i}-1} \ \ \ \text{and} \ \ \ \psi^{-1}(z)=\frac{(1-\mathbf{i})z-1}{z-\mathbf{i}-1}.
\]
We will prove the spectrum of $C_{\psi}$ acting on $H^2_{\beta}$ is the unit circle divided into three steps.

$\bf (1)$  The spectrum of $C_{\psi}$ acting on $H^2_{\beta}$ is contained in the unit circle.

By  Theorem \ref{NormC}, we have for any positive integer $n$, there exists a positive number $K$ such that
\[
\|C^n_{\psi}\|_{H^2_{\beta}}=\|C_{\psi^n}\|_{H^2_{\beta}}\leq \frac{K}{(1-|\frac{n}{n+\mathbf{i}}|)^{2N^2+3N+2}},
\]
and
\[
\|C^{-n}_{\psi}\|_{H^2_{\beta}}=\|C_{\psi^{-n}}\|_{H^2_{\beta}}\leq \frac{K}{(1-|\frac{n}{n-\mathbf{i}}|)^{2N^2+3N+2}}.
\]
Consequently
\[
\rho(C_{\psi})=\lim\limits_{n\rightarrow+\infty}\|C^n_{\psi}\|^{\frac{1}{n}}_{H^2_{\beta}}\leq 1,
\]
and
\[
\rho(C^{-1}_{\psi})=\lim\limits_{n\rightarrow+\infty}\|C^{-n}_{\psi}\|^{\frac{1}{n}}_{H^2_{\beta}}\leq 1.
\]
Therefore,
\[
\sigma(C_{\psi})\subseteq \partial\mathbb{D}.
\]

$\bf (2)$  The spectrum of $C_{\psi}$ acting on $H^2_{\beta}$ is either the whole unit circle or the single point set $\{1\}$.

By Remark \ref{reparabolic} and Lemma \ref{similar}, for any positive integer $n$, $C_{\psi}$ is similar to $C^n_{\psi}$ on $H^2_{\beta}$ and consequently,
\[
\sigma(C_{\psi})=\sigma(C^n_{\psi}),
\]
that means if $\lambda\in\sigma(C_{\psi})$, then $\lambda^n\in\sigma(C_{\psi})$.

Now, suppose that the spectrum of $C_{\psi}$ on $H^2_{\beta}$ is not the whole unit circle. Denote $\lambda=\mathrm{e}^{t2\pi\mathbf{i}}$, $t\in\mathbb{R}$. Then, if $t$ is a irrational number, $\lambda$ is not in $\sigma(C_{\psi})$. Because if $\lambda$ is in $\sigma(C_{\psi})$, then
\[
\sigma(C_{\psi})\supseteq\overline{\{\lambda^n; \ n\in\mathbb{N}\}}=\partial\mathbb{D}.
\]

If $\lambda\in\sigma(C_{\psi})$ and $t=\frac{q}{p}$ is a rational number, where $p$ and $q$ are two coprime positive integers, then
\[
\{1,\mathrm{e}^{\frac{1}{p}\cdot2\pi\mathbf{i}},\mathrm{e}^{\frac{2}{p}\cdot2\pi\mathbf{i}},\ldots,\mathrm{e}^{\frac{p-1}{p}\cdot2\pi\mathbf{i}}\}\subseteq\sigma(C_{\psi}).
\]
Since $\sigma(C_{\psi})$ is a closed proper subset of $\partial\mathbb{D}$ which contains no irrational points, $\sigma(C_{\psi})$ is a finite set and consequently we could write $\sigma(C_{\psi})$ as follows
\[
\sigma(C_{\psi})=\{1, \mathrm{e}^{\frac{1}{p_1}\cdot2\pi\mathbf{i}},\ldots,\mathrm{e}^{\frac{p_1-1}{p_1}\cdot2\pi\mathbf{i}},\ldots,\mathrm{e}^{\frac{1}{p_M}\cdot2\pi\mathbf{i}},\ldots,
\mathrm{e}^{\frac{p_M-1}{p_M}\cdot2\pi\mathbf{i}}\},
\]
where $M$ is positive integer. Furthermore,
\[
\sigma(C_{\psi})=\sigma(C^{p_1p_2\cdots p_M}_{\psi})=\{\lambda^{p_1p_2\cdots p_M}; \lambda\in\sigma(C_{\psi})\}=\{1\}.
\]
Therefore, if $\sigma(C_{\psi})\subsetneqq\partial\mathbb{D}$, we have $\sigma(C_{\psi})=\{1\}$.

$\bf (3)$  The spectrum of $C_{\psi}$ acting on $H^2_{\beta}$ is just the unit circle.

Suppose that $\sigma(C_{\psi})=\{1\}$. We also have $\sigma(C_{\psi^{-1}})=\sigma(C^{-1}_{\psi})=\{1\}$, and then for any $\epsilon>0$, when $m\in\mathbb{N}$ is large enough,
\[
\|C_{\psi^{-m}}\|_{H^2_{\beta}}=\|C^{-m}_{\psi}\|_{H^2_{\beta}}\le(1+\epsilon)^m.
\]
In particular,
\[
\|\psi^{-2m}(z)\|_{H^2_{\beta}}=\|C_{\psi^{-2m}}(z)\|_{H^2_{\beta}}\le(1+\epsilon)^{2m}.
\]

Consider the operator $(C_{\psi}-\mathbf{I})^{4m}$, where $\mathbf{I}$ is the identity operator. One can see
\[
(C_{\psi}-\mathbf{I})^{4m}(\psi^{-2m})=\sum\limits_{n=0}^{4m}(-1)^nC_{4m}^{n}\cdot C_{\psi^{n}}(\psi^{-2m})=\sum\limits_{n=0}^{4m}(-1)^nC_{4m}^{n}\cdot \psi^{n-2m}.
\]
Notice that for any $k\in\mathbb{N}$,
\[
\psi^{k}(0)=\frac{k}{k-\mathbf{i}}=\frac{k^2}{k^2+1}+\mathbf{i}\frac{k}{k^2+1} \ \ \text{and} \ \ \psi^{-k}(0)=\frac{k}{k+\mathbf{i}}=\frac{k^2}{k^2+1}-\mathbf{i}\frac{k}{k^2+1}.
\]
Then,
\begin{align*}
\mathrm{Re}((C_{\psi}-\mathbf{I})^{4m}\psi^{-2m}(0))=&\mathrm{Re}(\sum\limits_{n=0}^{4m}(-1)^nC_{4m}^{n}\cdot \psi^{n-2m}(0)) \\
=&\sum\limits_{n=0}^{4m}(-1)^nC_{4m}^{n}\cdot\frac{(n-2m)^2}{(n-2m)^2+1}.
\end{align*}
Furthermore, we have
\begin{align*}
&-\mathrm{Re}((C_{\psi}-\mathbf{I})^{4m}\psi^{-2m}(0))=(1-1)^{4m}-\mathrm{Re}((C_{\psi}-\mathbf{I})^{4m}\psi^{-2m}(0)) \\
=&\sum\limits_{n=0}^{4m}(-1)^nC_{4m}^{n}-\sum\limits_{n=0}^{4m}(-1)^nC_{4m}^{n}\cdot\frac{(n-2m)^2}{(n-2m)^2+1} \\
=&\sum\limits_{n=0}^{4m}(-1)^nC_{4m}^{n}\cdot\frac{1}{(n-2m)^2+1}.
\end{align*}
One can see that
\[
\sum\limits_{n=2m-1}^{2m+1}(-1)^nC_{4m}^{n}\cdot\frac{1}{(n-2m)^2+1}=-\frac{1}{2}C_{4m}^{2m-1}+C_{4m}^{2m}-\frac{1}{2}C_{4m}^{2m+1}>0,
\]
and for $k=1,2,\ldots,m-1$,
\[
C_{4m}^{2k}\cdot\frac{1}{(2k-2m)^2+1}-C_{4m}^{2k-1}\cdot\frac{1}{(2k-1-2m)^2+1}>0,
\]
and hence
\begin{align*}
&\sum\limits_{n=1}^{2m-2}\left((-1)^nC_{4m}^{n}\cdot\frac{1}{(n-2m)^2+1}\right) \\
=&\sum\limits_{k=1}^{m-1}\left(-C_{4m}^{2k-1}\cdot\frac{1}{(2k-1-2m)^2+1}+C_{4m}^{2k}\cdot\frac{1}{(2k-2m)^2+1}\right) \\
>& 0.
\end{align*}
In addition, for $n=0,1,\ldots,4m$,
\[
(-1)^nC_{4m}^{n}\cdot\frac{1}{(n-2m)^2+1}=(-1)^{4m-n}C_{4m}^{4m-n}\cdot\frac{1}{((4m-n)-2m)^2+1},
\]
and hence
\[
\sum\limits_{n=2m+1}^{4m-1}\left((-1)^nC_{4m}^{n}\cdot\frac{1}{(n-2m)^2+1}\right)=\sum\limits_{n=1}^{2m-2}\left((-1)^nC_{4m}^{n}\cdot\frac{1}{(n-2m)^2+1}\right)>0.
\]
Then,
\begin{align*}
&-\mathrm{Re}((C_{\psi}-\mathbf{I})^{4m}\psi^{-2m}(0)) \\
=&\frac{1}{(2m)^2+1}+\left(\sum\limits_{n=1}^{2m-2}+\sum\limits_{n=2m-1}^{2m+1}+\sum\limits_{n=2m+1}^{4m-1}\right)
\left((-1)^nC_{4m}^{n}\cdot\frac{1}{(n-2m)^2+1}\right) \\
&+\frac{1}{(2m)^2+1} \\
\ge&\frac{2}{4m^2+1}.
\end{align*}
Consequently,
\begin{align*}
&\left\|(C_{\psi}-\mathbf{I})^{4m}\frac{\psi^{-2m}}{\|\psi^{-2m}\|_{H^2_{\beta}}}\right\|_{H^2_{\beta}} \\
\ge&\frac{1}{\|\psi^{-2m}\|_{H^2_{\beta}}}\cdot |(C_{\psi}-\mathbf{I})^{4m}\psi^{-2m}(0)| \\
\ge&\frac{1}{\|\psi^{-2m}\|_{H^2_{\beta}}}\cdot |-\mathrm{Re}((C_{\psi}-\mathbf{I})^{4m}\psi^{-2m}(0))| \\
\ge& \frac{1}{(1+\epsilon)^{2m}}\cdot \frac{2}{4m^2+1}.
\end{align*}
Therefore,
\begin{align*}
\rho(C_{\psi}-\mathbf{I})=&\lim\limits_{m\rightarrow\infty}\|(C_{\psi}-\mathbf{I})^{4m}\|_{H^2_{\beta}}^{\frac{1}{4m}} \\
\ge&\lim\limits_{m\rightarrow\infty}\left\|(C_{\psi}-\mathbf{I})^{4m}\frac{\psi^{-2m}}{\|\psi^{-2m}\|_{H^2_{\beta}}}\right\|_{H^2_{\beta}}^{\frac{1}{4m}} \\
\ge&\lim\limits_{m\rightarrow\infty}\left(\frac{1}{(1+\epsilon)^{2m}}\cdot \frac{2}{4m^2+1}\right)^{\frac{1}{4m}} \\
=&\frac{1}{\sqrt{1+\epsilon}} \\
>&0,
\end{align*}
which is a contradiction to $\sigma(C_{\psi}-\mathbf{I})=\{0\}$.

Now we have obtained that the spectrum of $C_{\psi}$ acting on $H^2_{\beta}$ is the unit circle. Furthermore,
\[
\sigma(C_{\psi^{-1}})=\sigma(C^{-1}_{\psi})=\{\frac{1}{z}; \ z\in\sigma(C_{\psi})\}=\partial\mathbb{D}.
\]
Thus, for any parabolic disk automorphism $\varphi$, the spectrum of $C_{\varphi}$ acting on $H^2_{\beta}$ is the unit circle.

\end{proof}

\begin{theorem}\label{Sh}
Let $H^2_{\beta}$ be a weighted Hardy space of polynomial growth induced by the weight sequence $\beta=\{\beta_k\}_{k=1}^{\infty}$.  If $\varphi$ is a hyperbolic disk automorphism, then the spectrum of $C_{\varphi}$ acting on $H^2_{\beta}$ is contained in an annulus. More precisely,
\[
\sigma(C_{\varphi})\subseteq \{z\in \mathbb{C}; \left(\frac{1-|\varphi(0)|}{1+|\varphi(0)|}\right)^{2N^2+3N+2} \leq|z|\leq \left(\frac{1+|\varphi(0)|}{1-|\varphi(0)|}\right)^{2N^2+3N+2} \}.
\]
\end{theorem}

\begin{proof}
Suppose that $H^2_{\beta}$ is a weighted Hardy space of $N$-polynomial growth.
By Lemma \ref{hyperboic}, it suffices to consider the spectrum of
\[
\psi(z)=\frac{z+r}{1+rz},
\]
where $r=|\varphi(0)|$.
After a straight computation, one can see
\[
\psi^n(z)=\underset{n}{\underbrace{\psi\circ\cdots\circ\psi}}(z)=\frac{1+\frac{(1+r)^n-(1-r)^n}{(1+r)^n+(1-r)^n}}{1+\frac{(1+r)^n-(1-r)^n}{(1+r)^n+(1-r)^n}z},
\]
Then, by  Theorem \ref{NormC}, there exists a positive number $K$ such that
\[
\|C^n_\psi\|_{H^2_{\beta}}=\|C_{\psi^n}\|_{H^2_{\beta}}\leq \frac{K}{\left(1-\frac{(1+r)^n-(1-r)^n}{(1+r)^n+(1-r)^n}\right)^{2N^2+3N+2}},
\]
and consequently
\[
\rho(C_\psi)=\lim\limits_{n\rightarrow+\infty}\|C^n_\psi\|^{\frac{1}{n}}_{H^2_{\beta}}\leq \left(\frac{1+r}{1-r}\right)^{2N^2+3N+2}.
\]
Similarly, we also obtain
\[
\rho(C^{-1}_\psi)\leq \left(\frac{1+r}{1-r}\right)^{2N^2+3N+2}.
\]
Therefore,
\[
\sigma(C_{\varphi})\subseteq \left\{z\in \mathbb{C}; \left(\frac{1-|\varphi(0)|}{1+|\varphi(0)|}\right)^{2N^2+3N+2} \leq|z|\leq \left(\frac{1+|\varphi(0)|}{1-|\varphi(0)|}\right)^{2N^2+3N+2} \right\}.
\]
\end{proof}

\section{Composition operators with closed ranges and Fredholmness}

In this section, we completely characterize the closed range and Fredholm composition operators with symbols of analytic self-maps on the closed unit disk acting on weighted Hardy spaces of polynomial growth.

\begin{proposition}\label{cr}
Let $H^2_{\beta}$ be a weighted Hardy space of polynomial growth. Let $\varphi:\overline{\mathbb{D}}\rightarrow\overline{\mathbb{D}}$ be a nontrivial analytic map. If the composition operator $C_{\varphi}$ acting on $H^2_{\beta}$ has closed range, then $\mathbb{T}\subseteq\varphi(\overline{\mathbb{D}})$.
\end{proposition}
\begin{proof}
Suppose that $H^2_{\beta}$ is a weighted Hardy space of $N$-polynomial growth.
If the composition operator $C_{\varphi}$ acting on $H^2_{\beta}$ has closed range, then there exists a positive number $c$ such that
\begin{equation}\label{rest}
\|C_{\varphi}(h)\|_{H^2_{\beta}} \ge c\|h\|_{H^2_{\beta}}, \ \ \ \ \ \text{for any} \ h\in H^2_{\beta}.
\end{equation}
Suppose that $\varphi(\overline{\mathbb{D}})$ dose not contain the unit circle $\mathbb{T}$. Then there exists a point $\xi\in\mathbb{T}$ and a positive number $\epsilon<\frac{1}{2}$ such that
\[
B(\xi, \epsilon)\bigcap \varphi(\overline{\mathbb{D}})=\emptyset,
\]
where $B(\xi, \epsilon)=\{z\in\mathbb{C}; \ |z-\xi|<\epsilon\}$.
Furthermore, we could select a positive number $t\in(0,1)$ such that the arc $J$ which lies in the unit disk, passes through the point $t\xi$ and is orthogonal to the unit circle, is contained in $B(\xi, \epsilon)$.

Let $\eta(z)=\frac{1+z}{1-z}$ be the map from the unit disk $\mathbb{D}$ to the right half-plane $\mathbb{H}_r$, and let $\varsigma(z)=\frac{(1-t)^2}{(1+t)^2}\cdot z$ be the map from $\mathbb{H}_r$ on to itself. Define $\tau : \mathbb{D}\rightarrow\mathbb{D}$ by
\[
\tau=\eta^{-1}\circ\varsigma\circ\eta.
\]
Then $\eta$ is a disk automorphism with fixed point $-\xi$ and maps $J$ to $-J$, where $-J=\{-z;z\in J\}$. Moreover, $\eta$ maps the region  bounded by $J$ and the arc in the unit circle passing through $-\xi$ to the region  bounded by $-J$ and the arc in the unit circle passing through $-\xi$ (see Figure \ref{tau}).
\begin{figure}[htbp]
	\centering
	\includegraphics[width=1\textwidth]{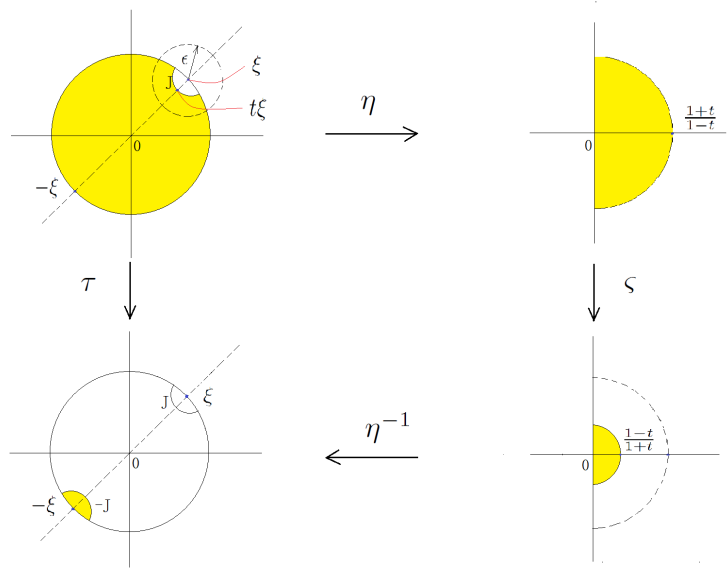}
	\caption{Disk automorphism $\tau$}
	\label{tau}
\end{figure}

Let $g(z)=z+\xi$. Consider the analytic function $f:\mathbb{D}\rightarrow\mathbb{D}$ defined by
\[
f(z)=g\circ\tau(z).
\]
Then $f\circ\varphi$ maps the unit disk into the disk $B(0, \epsilon)$ (more precisely, $f(\mathbb{D})$ is contained in the yellow region as following Figure \ref{fimage}.).
\begin{figure}[htbp]
	\centering
	\includegraphics[width=1\textwidth]{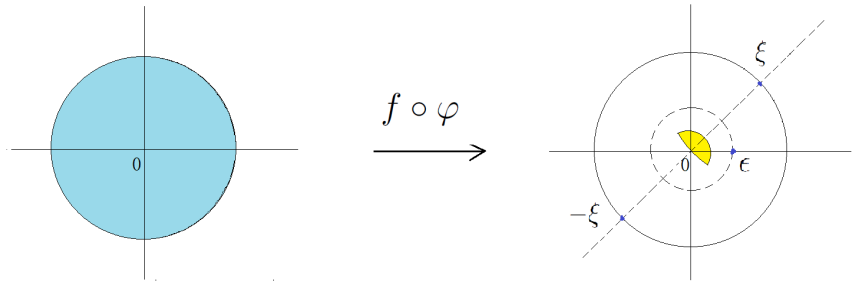}
	\caption{Analytic function $f\circ\varphi$}
	\label{fimage}
\end{figure}

Now set $f_k(z)=(f(z))^k$, for $k=1,2,\ldots$.  We write
\[
F=C_{\varphi}(f)=f\circ\varphi, \ \ \ \ F_k(z)=(F(z))^k
\]
and
\[
C_{\varphi}(f_k(z))=F_k=\sum\limits_{n=0}^{\infty}\widehat{F_k}(n)z^n.
\]
Suppose that $\varphi$ is analytic on a disk $B(0,R)$, for some certain $R>1$. Then each $F_k$ is also analytic on $B(0,R)$.

Since the maximum modulus of $f_k(z)$ on $\mathbb{D}$ is less than $\epsilon^k$, it follows from Cauchy inequality that
\[
|\widehat{F_k}(n)|\le \frac{\epsilon^{k}}{R^n},
\]
and consequently as $k\rightarrow\infty$,
\[
\|C_{\varphi}(f_k)\|_{H^2_{\beta}}=\sqrt{\sum\limits_{n=0}^{\infty}|\widehat{F_k}(n)|^2\beta_n^2} \le\sqrt{\sum\limits_{n=0}^{\infty}\frac{\epsilon^{2k}\beta_n^2}{R^{2n}}}=\epsilon^k\sqrt{\sum\limits_{n=0}^{\infty}\frac{\beta_n^2}{R^{2n}}}\rightarrow0.
\]

Notice that for every $k$,
\[
\|(g(z))^k\|_{H^2_{\beta}}\ge|(g(0))^k|=|\xi^k|=1.
\]
In addition, we have known that $C_{\tau}$ is a bounded invertible operator acting on $H^2_{\beta}$ by Theorem \ref{Mz}, i.e., there exist positive numbers $a$ and $A$ such that
\[
a\|h\|_{H^2_{\beta}}\le \|C_{\tau}(h)\|_{H^2_{\beta}} \le A\|h\|_{H^2_{\beta}}, \ \ \ \ \ \text{for any} \ h\in H^2_{\beta}.
\]
Then for all $k$,
\[
\|f_k\|_{H^2_{\beta}} = \|C_{\tau}((g(z))^k)\|_{H^2_{\beta}} \ge a\|(g(z))^k\|_{H^2_{\beta}} \ge a >0.
\]
This is a contradiction to the inequality (\ref{rest}).
\end{proof}

To continue our study, some results related to $0$-dimensional spaces are required (we refer to \cite{H41}).
\begin{definition}[Definition II 1 in \cite{H41} p.10]
A separable metric space $X$ has dimension $0$ at a point $p$ if $p$ has arbitrarily small neighborhood with empty boundaries. A nonempty space $X$ has $0$-dimension if it has dimension $0$ at each of its points.
\end{definition}

\begin{theorem}\label{0dim}
All spaces considered are separable metric. Then
\begin{enumerate}
\item \ The property of being $0$-dimensional is a topological invariant, i.e., if the space $X$ is homeomorphic to a $0$-dimensional space, then $X$ is also $0$-dimensional. {\rm((A) in \cite{H41} p.10)}
\item \ A nonempty subset of a $0$-dimensional space is $0$-dimensional. {\rm(Theorem II 1 in \cite{H41} p.13)}
\item \ A space which is the sum of countable $0$-dimensional closed subsets is itself $0$-dimensional. {\rm(Theorem II 2 in \cite{H41} p.18)}
\item \ Any set of real numbers containing no interval is $0$-dimensional. {\rm(Example II 4 in \cite{H41} p.11)}
\item \ A connected $0$-dimensional consists of only one point. {\rm((C) in \cite{H41} p.15)}
\end{enumerate}
\end{theorem}

\begin{theorem}\label{semiFr}
Let $H^2_{\beta}$ be a weighted Hardy space of polynomial growth induced by the weight sequence $\beta=\{\beta_k\}_{k=1}^{\infty}$. Let $\varphi:\overline{\mathbb{D}}\rightarrow\overline{\mathbb{D}}$ be a nontrivial analytic map. Then the followings are equivalent.
\begin{enumerate}
\item \ The composition operator $C_{\varphi}$ is semi-Fredholm.
\item \ The composition operator $C_{\varphi}$ has closed range.
\item \ $\mathbb{T}\subseteq\varphi(\overline{\mathbb{D}})$.
\item \ $\varphi(\mathbb{T})$ contains an arc in $\mathbb{T}$.
\item \ $\varphi$ is a finite Blaschke product.
\end{enumerate}
\end{theorem}

\begin{proof}
Obviously, $(1)$ implies $(2)$ and $(3)$ implies $(4)$. By Proposition \ref{cr}, one can see that $(2)$ implies $(3)$. By Theorem \ref{MB}, one can see that $(5)$ implies $(1)$, more precisely, the composition operator $C_{\varphi}$ with a symbol of finite Blaschke product with order $m$ is an injective operator with closed range, and has codimension $1$ if $m=1$, and has infinite codimension if $m>1$.
So it suffices to prove $(4)$ implies $(5)$.

Suppose that $\varphi(\mathbb{T})$ contains an arc in $\mathbb{T}$.
Since $\varphi$ is analytic on $\overline{\mathbb{D}}$, we could write $\varphi=BF$, where $B$ is a finite Blaschke product and $F$ is an outer function which is also analytic on $\overline{\mathbb{D}}$ (see \cite{Gar81} for exmple). Notice that $|B(\xi)|=1$ for all $\xi\in\mathbb{T}$. Then $F(\mathbb{T})$ contains a closed arc, denoted by $J_1$, in $\mathbb{T}$.
Obviously, $F^{-1}(J_1)$ is a closed subset in $\mathbb{T}$.

Now we claim that $F^{-1}(J_1)$ is a closed subset with nonempty interior in $\mathbb{T}$. Suppose that the closed set $F^{-1}(J_1)$  has empty interior in $\mathbb{T}$. Then by Theorem \ref{0dim}, $F^{-1}(J_1)$ is homeomorphic to a set of real numbers containing no interval and hence is $0$-dimensional. Since $F(z)$ is analytic on $\overline{\mathbb{D}}$, $F$ is analytic on a disc $B(0,R)$ for some $R>1$. To avoid confusion, let $F_{R}$ be the analytic function on $B(0,R)$ such that $F_R(z)=F(z)$ for any $z\in\overline{\mathbb{D}}$.
Notice that there is at most finite zero points of $F'_R(z)$(or written as $F'(z)$) on $\mathbb{T}$. Then we could choose a sub-arc $J_2$ of $J_1$ such that $F'(\lambda)\neq0$ for any $\lambda\in F^{-1}(J_2)$. Furthermore, one can select a point $\xi_0\in J_2$ such that
\[
\textrm{Card}\{F^{-1}(\xi_0)\}=\textrm{Card}\{\lambda\in\mathbb{T}; ~F(\lambda)=\xi_0\}=\max\limits_{\xi\in J_2}~\textrm{Card}\{\lambda\in\mathbb{T}; ~F(\lambda)=\xi\}.
\]
Denote $F^{-1}(\xi_0)=\{\lambda_j, j=1,2,\ldots,n_0.\}$, where $n_0=\textrm{Card}\{\lambda\in\mathbb{T}; ~F(\lambda)=\xi_0\}$. Following from $F'_R(\lambda_j)\neq 0$ for $j=1,2,\ldots,n_0$, there exists a neighborhood $V$ of $\xi_0$ and a neighborhood $U_j$ of $\lambda_j$ for all $j$ such that the restriction of $F_R$ on $U_j$ is an analytic homeomorphism to $V$. Denote $J_3$ be a closed arc in $J_2\bigcap V$, and moreover denote $A_j=F^{-1}(J_3)\bigcap U_j$ for $j=1,2,\ldots,n_0$. By Theorem \ref{0dim}, each $A_j$ is $0$-dimensional and consequently
\[
J_3=F_R(F^{-1}(J_3))=F_R(\bigcup\limits_{j=1}^{n_0}A_j)=\bigcup\limits_{j=1}^{n_0}F_R(A_j)
\]
is also $0$-dimensional. However, by $(5)$ in Theorem \ref{0dim}, the arc $J_3$ is not $0$-dimensional. This is a contradiction.

So, $F^{-1}(J_1)$ is a closed subset with nonempty interior in $\mathbb{T}$. Furthermore, there exists an arc $J$ in $F^{-1}(J_1)\subseteq\mathbb{T}$ and then $F(J)=\widetilde{J}$ is a sub-arc in $J_1\subseteq\mathbb{T}$.

In addition, $F(\mathbb{D})$ is contained in $\mathbb{D}$ and does not contain zero. Then by the Schwartz reflection principle, one can obtain an analytic function $\widetilde{F}$ on the domain $\Omega=J\bigcup(\mathbb{C}\setminus\mathbb{T})$, which is an analytic continuation of $F$ on $\mathbb{D}$.

Since $F$ is analytic on $\overline{\mathbb{D}}$ and coincide with $\widetilde{F}$ on $\mathbb{D}$, it follows the uniqueness of analytic functions that $F$ can be analytically extended to be an entire function on $\mathbb{C}$. Moreover, one can see that $F(\overline{\mathbb{D}})$ also does not contain zero. If $0\in F(\overline{\mathbb{D}})$, there exists a sequence of points $z_n\in\mathbb{D}$ tending to a point $\xi\in\mathbb{T}$ such that $F(z_n)\rightarrow 0$. By the Schwartz reflection principle, we have $\widetilde{F}(\frac{1}{\overline{z_n}})\rightarrow \infty$. However, this can not occur because of $\frac{1}{\overline{z_n}}\rightarrow \xi\in\mathbb{T}$. Then, $F$ can be analytically extended to be a bounded entire function, and consequently $F$ must be a constant function with module $1$. Therefore, $\varphi$ is just a finite Blaschke product.
\end{proof}

\begin{theorem}\label{Fr}
Let $H^2_{\beta}$ be a weighted Hardy space of polynomial growth.  Let $\varphi:\overline{\mathbb{D}}\rightarrow\overline{\mathbb{D}}$ be an analytic map. Then the composition operator $C_{\varphi}$ acting on $H^2_{\beta}$ is Fredholm if and only if $\varphi$ is a disk automorphism.
\end{theorem}

\begin{proof}
Notice that for a finite Blaschke product $B$ with order $m$, by Theorem \ref{MB} and Remark \ref{BsemiF}, the composition operator $C_{B}$ has finite codimension if and only if the order $m$ equals $1$. Then, by Theorem \ref{semiFr}, the composition operator $C_{\varphi}$ is Fredholm if and only if $\varphi$ is a disk automorphism.
\end{proof}

{\bf Acknowledgement}

The second author was supported by National Natural Science Foundation of China (Grant No. 11831006, 11920101001 and 11771117). The authors
thank Professor Guangfu Cao and Professor Maofa Wang for many valuable discussions on this paper.

\end{document}